\documentclass[11pt,amssymb]{article}
\usepackage{eucal}
\usepackage{amsfonts,amsmath,amssymb,latexsym}
\usepackage{geometry}

\usepackage{graphicx}
\usepackage{mathptmx} 
\usepackage{eucal}
\usepackage{amsfonts,amsmath,amssymb,latexsym}
\usepackage{geometry}
\usepackage[utf8]{inputenc} 
\usepackage{latexsym}

\newtheorem{theorem}{Theorem}
\newtheorem{proposition}{Proposition}

\newtheorem{lemma}{Lemma}
\newtheorem{definition}{Definition}
\newtheorem{remark}{Remark}

\newtheorem{proof}{Proof}
\numberwithin{equation}{section}

\newcommand{\ds}{\displaystyle}
\newcommand{\al}{\alpha}
\newcommand{\be}{\beta}
\newcommand{\ga}{\gamma}

\newcommand{\x}{\textbf{X}}
\newcommand{\w}{\omega}
\newcommand{\E}{\tilde{E}}
\newcommand{\e}{\overline{E}}
\newcommand{\wa}{\omega_{\alpha}}
\newcommand{\wb}{\omega_\beta}
\newcommand{\wg}{\omega_\gamma}
\newcommand{\eg}{e_{\gamma}}
\newcommand{\ea}{e_{\alpha}}
\newcommand{\eb}{e_{\beta}}
\newcommand{\vg}{\varepsilon_{\gamma}}
\newcommand{\va}{\varepsilon_{\alpha}}
\newcommand{\vb}{\varepsilon_{\beta}}
\newcommand{\pe}{\left<}
\newcommand{\pd}{\right>}
\newcommand{\lp}{\left(}
\newcommand{\rp}{\right)}
\newcommand{\we}{\wedge}
\newcommand{\ce}{structure equations}
\newcommand{\nf}{\nabla^{\atop F}}

\newcommand{\nt}{\tilde{\nabla}}

\begin{document}

\title{A fundamental theorem for submanifolds in semi-Riemannian warped products}

\author{Carlos A. D. Ribeiro and Marcos F. de Melo}



\maketitle

\begin{abstract}
In this paper we find necessary and sufficient conditions for a nondegenerate arbitrary signature manifold $M^n$ to be realized as a submanifold in the large class of warped product manifolds $\varepsilon I\times_a\mathbb{M}^{N}_{\lambda}(c)$, where $\varepsilon=\pm 1,\ a:I\subset\mathbb{R}\to\mathbb{R}^+$ is the scale factor and $\mathbb{M}^{N}_{\lambda}(c)$ is the $N$-dimensional semi-Riemannian space form of index $\lambda$ and constant curvature $c\in\{-1,1\}.$  We prove that if $M^n$ satisfies Gauss, Codazzi and Ricci equations for a submanifold in $\varepsilon I\times_a\mathbb{M}^{N}_{\lambda}(c)$, along with some additional conditions, then $M^n$ can be isometrically immersed into $\varepsilon I\times_a\mathbb{M}^{N}_{\lambda}(c)$. This comprises the case of hypersurfaces immersed in semi-Riemannian warped products proved by M.A. Lawn and M. Ortega (see \cite{L&O}), which is an extension of the isometric immersion result obtained by J. Roth in the Lorentzian products $\mathbb{S}^n\times\mathbb{R}_1$ and $\mathbb{H}^n\times\mathbb{R}_1$ (see \cite{JR}), where $\mathbb{S}^n$ and $\mathbb{H}^n$ stand for the sphere and hyperbolic space of dimension $n$, respectively.  This last result, in turn, is an expansion to pseudo-Riemannian manifolds of the isometric immersion result proved by B. Daniel in $\mathbb{S}^n\times\mathbb{R}$ and $\mathbb{H}^n\times\mathbb{R}$ (see \cite{BD}), one of the first generalizations of the classical theorem for submanifolds in space forms (see \cite{KT}). Although additional conditions to Gauss, Codazzi and Ricci equations are not necessary in the classical theorem for submanifolds in space forms, they appear in all other cases cited above.



\end{abstract}

\section{Introduction} \label{intro}

A fundamental question in the theory of submanifolds is to know when a (pseudo)-Riemannian manifold can be isometrically immersed in a given ambient space. In the case where the ambient space is a Riemannian space form, it is a well-known fact that the Gauss, Codazzi and Ricci
equations are necessary and sufficient (see \cite{KT}). This fact is also true for psuedo-Riemannian manifolds, and a short proof of it can be found in \cite{LTV}, where the authors there used it to obtain isometric immersion results in product of space forms. In particular for immersions of codimension 1,
the Ricci equation is trivial and the Gauss and Codazzi equations are equivalent to the existence of a local isometric immersion into the desired space form.

Versions of the fundamental theorem of submanifold theory were recently achieved by Benoit Daniel in \cite{BD,BD2}, Julien Roth in \cite{JR} and by Marie-Amélie Lawn and Miguel Ortega in \cite{L&O}. In \cite{BD}, Daniel considers hypersurfaces in $\mathbb{M}^n\times\mathbb{R}$, where $\mathbb{M}^n$ is the sphere $\mathbb{S}^n$ or the hyperbolic space $\mathbb{H}^n$ of dimension $n$. He observed that for a hypersurface $M$ of $\mathbb{M}^n\times\mathbb{R}$ the Gauss and Codazzi equations depend only on the first and second fundamental form, the projection $T$ of the vertical vector $\partial t$ onto the tangent bundle $TM$ and its normal component $\nu$, and proved that these two equations together with additional first order differential equations in $T$ and $\nu$ give necessary and sufficient conditions for a manifold $M$ to be immersed locally and isometrically into $\mathbb{M}^n\times\mathbb{R}$. In \cite{JR}, Roth expanded Daniel's result and proved a fundamental theorem of hypersurfaces for the case of the Lorentzian products $\mathbb{M}^n\times\mathbb{R}_1$. In \cite{L&O}, Lawn and Ortega generalized Roth's result and obtained a fundamental theorem for hypersurfaces in semi-Riemannian warped products. It is interesting to bring to attention that none of these three results already mentioned above deal with submanifolds of codimension greater than one. It is also worth to point out that theorems of submanifold for the case of arbitrary codimension have been proved (see, e.g., \cite{DK,L&R,L&Z,L&M,LTV}).

Our aim in this paper is to give a proof of a fundamental theorem for submanifolds in semi-Riemannian warped products. It extends the result of Lawn and Ortega. Namely, we consider isometric immersions of arbitrary codimensions. Like the proof in \cite{L&O} we use moving frames and integrable distributions, the standard method based on differential forms first used by Cartan and after Tenenblat \cite{KT} or Daniel \cite{BD} for instance.

\section{Preliminaries} \label{Preliminaries}

Let $(P^m,g_P)$ be a semi-Riemannian manifold of dimension $\dim P=m$. All our manifolds will be connected and of class $C^{\infty}$, unless otherwise stated. We consider a smooth function $a:I\subset\mathbb{R}\to\mathbb{R}^+$, a (sign) constant $\varepsilon=\pm 1$ and the warped product
\[
\bar{P}^{m+1} = \varepsilon I\times_aP^m, \hspace{1cm} \langle,\ \rangle = \varepsilon dt^2+a(t)^2g_P
\] 
We set our convention for the curvature operator $\mathcal{R}$ of a connection $\mathcal{D}$ as 
\[
\mathcal{R}(X,Y)Z = \mathcal{D}_X\mathcal{D}_YZ - \mathcal{D}_Y\mathcal{D}_XZ - \mathcal{D}_{[X,Y]}Z.
\] 
Let $\bar{R}_P$ and $R_P$ be the curvature operator of 
$\bar{P}^{m+1}$ and $P^m$, respectively. Since there is no confusion, we will use the same notation for the associated curvature tensors along the paper. Let $\bar{\nabla}$ and $\nabla^{P}$ be the Levi-Civita connection of $\bar{P}^{m+1}$ and $P^{m}$, respectively. We recall the following formulae, which can be checked in \cite{BO}, bearing in mind the change of sign.

\begin{lemma} 
\label{Warped derivative rules}
(Cf. \cite{BO}) On the semi-Riemannian manifold $\bar{P}^{m+1}$, the following statements hold, for any $V, W$ lifts of vector fields tangent to $P^m$:
\begin{itemize}
\item[1.] $\bar{\nabla}_{\partial t}\partial t = 0,\ \bar{\nabla}_{V}\partial t = \frac{a'}{a}V,$
\item[2.] $\mathrm{grad}(a) = \varepsilon a' \partial t,$
\item[3.] $\bar{\nabla}_VW = \nabla^{P}_VW - \frac{\varepsilon a'}{a}\langle V,W\rangle \partial t,$
\item[4.] $\bar{R}_P(V,\partial t)\partial t = -\frac{a''}{a}V,\ \bar{R}_P(\partial t,V)W = \frac{\varepsilon a''}{a}\langle V,W\rangle \partial t,\ \bar{R}_P(V,W)\partial t = 0.$
\end{itemize}

\end{lemma}

We also need to set the standard Euclidean space of dimension $N+1\geq 3$ and index $\lambda$, $\mathbb{R}^{N+1}_{\lambda}$ with its metric 
\[
g_0 = \sum_{i=0}^{N-\lambda}dx_i^2 - \sum_{i=N-\lambda+1}^Ndx_i^2.
\]
Let $\mathbb{M}^N_{\lambda}(c)$ be the semi-Riemannian space form of constant sectional curvature $c\in\{-1,1\}$ and index $\lambda$, with metric $g$. We put
\[
\mathbb{E}^{N+1}_{\lambda} = 
\left\{
\begin{array}{l}
\mathbb{R}^{N+1}_{\lambda},\ \ \mathrm{if}\ \mathbb{M}^N_{\lambda}(c)=\mathbb{S}^N_{\lambda},\ c=+1,\\
\mathbb{R}^{N+1}_{\lambda+1}, \ \ \mathrm{if}\ \mathbb{M}^N_{\lambda}(c)=\mathbb{H}^N_{\lambda},\ c=-1,
\end{array}
\right.
\]
with its standard pseudo-Riemannian flat metric $g_0$ and Levi-Civita connection $\nabla^0.$ We are considering $\mathbb{S}^N_{\lambda}$ and $\mathbb{H}^N_{\lambda}$, respectively, as one of the connected components of
\[
\mathbb{S}^N_{\lambda}=\{p\in\mathbb{E}^{N+1}_{\lambda};\ g_0(p,p)=+1\}, \ \ \mathbb{H}^N_{\lambda}=\{p\in\mathbb{E}^{N+1}_{\lambda};\ g_0(p,p)=-1\}. 
\]
Following the previous notation, we consider $\bar{P}^{N+1}=\varepsilon I\times_a\mathbb{M}^N_{\lambda}(c)$ with metric $\langle,\ \rangle$, Levi-Civita connection $\bar{\nabla}$ and curvature tensor $\bar{R}$. We also consider $\tilde{P}^{N+2}=\varepsilon I\times_a\mathbb{E}^{N+1}_{\lambda}$ with metric $\langle,\ \rangle_2$, Levi-Civita connection $\tilde{\nabla}$ and curvature tensor $\tilde{R}.$
From the usual totally umbilical embedding $\Xi:\mathbb{M}^N_{\lambda}(c)\to\mathbb{E}^{N+1}_{\lambda}$, we construct the isometric immersion 
\[
\tilde{\Xi}:(\varepsilon I\times_a\mathbb{M}^N_{\lambda}(c),\langle,\ \rangle)\to(\varepsilon I\times_a\mathbb{E}^{N+1}_{\lambda},\langle,\ \rangle_2),\ \ (t,p)\mapsto (t,\Xi(p)).
\] 

\begin{proposition} \label{curvature tensor tilde{R}}
(Cf. \cite{BO}) The curvature tensor $\tilde{R}$ of $\tilde{P}^{N+2}$ is  given by  

\begin{eqnarray*}
\tilde{R}(X,Y,Z,W) &=& \varepsilon\frac{(a')^2}{a}\Big(\langle X,Z\rangle \langle Y,W\rangle - \langle Y,Z\rangle \langle X,W\rangle \Big)\\
                   && + \left(\frac{a''}{a}-\frac{(a')^2}{a^2}\right)\Big(\langle X,Z\rangle \langle Y,\partial t\rangle \langle W,\partial t \rangle  -  \langle Y,Z\rangle \langle X,\partial t\rangle \langle W,\partial t \rangle \Big.\\
                   && \Big. - \langle X,W\rangle \langle Y,\partial t\rangle \langle Z,\partial t \rangle  + \langle Y,W\rangle \langle X,\partial t\rangle \langle Z,\partial t \rangle  \Big),\\ 
\end{eqnarray*}
for any sections $X,Y,Z,W\in\Gamma(T\tilde{P}^{N+2})$.
\end{proposition}
Using the general Gauss equation for the isometric immersion above, one concludes the following: 
\begin{proposition} \label{curvature tensor bar{R}}
(Cf. \cite{BO}) The curvature tensor $\bar{R}$ of $\bar{P}^{N+1}$ is given by

\begin{eqnarray*}
\bar{R}(X,Y,Z,W) &=& \left(\varepsilon\frac{(a')^2}{a^2}-\frac{c}{a^2}\right) \Big( \langle X,Z\rangle \langle Y,W\rangle - \langle Y,Z\rangle \langle X,W\rangle \Big)\\
                 && + \left(\frac{a''}{a}-\frac{(a')^2}{a^2}+\frac{\varepsilon c}{a^2}\right)\Big(\langle X,Z\rangle \langle Y,\partial t\rangle \langle W,\partial t \rangle  -  \langle Y,Z\rangle \langle X,\partial t\rangle \langle W,\partial t \rangle \Big.\\
                   && \Big. - \langle X,W\rangle \langle Y,\partial t\rangle \langle Z,\partial t \rangle  + \langle Y,W\rangle \langle X,\partial t\rangle \langle Z,\partial t \rangle  \Big),\\ 
\end{eqnarray*}
for any tangent vector fields $X,Y,Z,W\in T\bar{P}^{N+1}.$\\

\end{proposition}

\section{Isometric immersions into warped products} \label{Submanifolds}
Let $M^n$ be a semi-Riemannian manifold isometrically immersed in $\bar{P}^{N+1}=\varepsilon I\times_a\mathbb{M}^N_{\lambda}(c)$. Let $\nabla$ and $\bar{\nabla}$ be the Levi-Civita connection of $M^n$ and $\bar{P}^{N+1}$, respectively. Denote by $R$ and $R^{\perp}$ the curvature tensors of the tangent and normal bundles $TM$ and $TM^{\perp}$, respectively, by $\alpha\in\Gamma(T^*M\otimes T^*M\otimes TM^{\perp})$ the second fundamental form of the immersion and by $A_{\eta}\in\Gamma(T^*M\otimes TM)$ its shape operator in the normal direction $\eta$, given by $\langle A_{\eta}X,Y \rangle = \langle \alpha(X,Y),\eta\rangle$, for all $X,Y\in\Gamma(TM)$. Let $\pi:M\to I$ be the restriction to $M$ of the natural projection $\pi_I:\varepsilon I\times_a\mathbb{M}^N_{\lambda}(c)\to I$ and $\partial t = T + \xi$, where $T=\varepsilon\cdot\textrm{grad}(\pi)$. Hence, we have $\xi\in\Gamma(TM^{\perp})$, because for any $x_0\in M$ and $v\in T_{x_0}M$
\[
\langle \xi(x_0), v\rangle  = \langle \partial t-T(x_0),v \rangle 
= \langle \partial t,v \rangle - \varepsilon d\pi_{x_0}(v) = 0.
\]
In particular, we have the equation
\[
\varepsilon = \langle \partial t,\partial t\rangle = \langle T,T \rangle + \langle \xi,\xi \rangle.
\]
In the following, we are using these decompositions for the tangent bundle $T\bar{P}^{N+1}$:

\[
T_{(t_0,x_0)}\bar{P}^{N+1} = T_{t_0}I\oplus T_{x_0}\mathbb{M}^N_{\lambda}(c) = 
T_{(t_0,x_0)}M^n\oplus (T_{(t_0,x_0)}M^n)^{\perp},\ \ (t_0,x_0)\in\bar{P}^{N+1}.
\]

\begin{lemma} \label{derivative rules in M immersed}
Under the previous conditions, the following equations hold for any $X\in TM$:
\begin{itemize}
\item[1.] $\bar{\nabla}_X\partial t = \frac{a'}{a}(X - \varepsilon\langle X,T \rangle \partial t).$
\item[2.] $\nabla_XT = \frac{a'}{a}(X-\varepsilon\langle X,T\rangle T) + A_{\xi}X.$
\item[3.] $\nabla_X^{\perp}\xi = -\frac{\varepsilon a'}{a}\langle X,T \rangle\xi - \alpha(X,T).$
\end{itemize}
\begin{proof}
For any $X\in TM$, there exists $X_0\in T\mathbb{M}^N_{\lambda}(c)$ such that $X = \varepsilon\langle X,T \rangle\partial t + X_0$. Using the Lemma \ref{Warped derivative rules}, one has
\[
\bar{\nabla}_X\partial t = \bar{\nabla}_{X_0}\partial t = \frac{a'}{a}X_0 = \frac{a'}{a}(X - \varepsilon\langle X,T \rangle \partial t).
\]
Finally, we obtain
\begin{eqnarray*}
\nabla_XT + \alpha(X,T) &=&  \bar{\nabla}_XT\\
                        &=& \bar{\nabla}_X\partial t - \bar{\nabla}_X\xi\\
                        &=& \frac{a'}{a}(X - \varepsilon\langle X,T \rangle \partial t) - (-A_{\xi}X+\nabla^{\perp}_X\xi)\\
                        &=& \frac{a'}{a}(X - \varepsilon\langle X,T \rangle T)+A_{\xi}X -\left(\frac{\varepsilon a'}{a}\langle X,T \rangle \xi+\nabla^{\perp}_X\xi\right).
\end{eqnarray*}
Now, the last two equation are just the tangential and the normal part of $\bar{\nabla}_XT.$ \hfill $\square$
\end{proof}
\end{lemma}
By using Proposition \ref{curvature tensor bar{R}} and the classical Gauss equation for submanifolds, which states
\[
R(X,Y,Z,W) = \bar{R}(X,Y,Z,W) - \langle \alpha(X,Z),\alpha(Y,W)\rangle + \langle \alpha(X,W),\alpha(Y,Z) \rangle,
\]
for any $X,Y,Z,W\in TM$, we obtain:
\begin{proposition} \label{Gauss for a given immersion}
The curvature tensor $R$ of $M$ in $\varepsilon I\times_a\mathbb{M}^N_{\lambda}(c)$ is
\begin{eqnarray*}
R(X,Y,Z,W) &=& \Bigg(\varepsilon\dfrac{(a')^2}{a^2}-\dfrac{c}{a^2}\Bigg)\Big(\left<X,Z\right>\left<Y,W\right>-\left<Y,Z\right>\left<X,W\right>\Big)\\
           && +\left(\dfrac{a''}{a}-\dfrac{(a')^2}{a^2}+\dfrac{\varepsilon c}{a^2}\right)\Big(\left<X,Z\right>\left<Y,T\right>\left<W,T\right>-\left<Y,Z\right>\left<X,T\right>\left<W,T\right>\Big.\\
           &&-\Big.\left<X,W\right>\left<Y,T\right>\left<Z,T\right>+\left<Y,W\right>\left<X,T\right>\left<Z,T\right>\Big)\\
           &&-\left<\alpha(X,Z),\alpha(Y,W)\right>+\left<\alpha(X,W),\alpha(Y,Z)\right>.\\
\end{eqnarray*}
\end{proposition}

\begin{proposition} \label{Codazzi for a given immersion}
The Codazzi equation of $M$ in $\varepsilon I\times_a\mathbb{M}^N_{\lambda}(c)$ is
\[
(\bar{\nabla}_Y\alpha)(X,Z,\eta) - (\bar{\nabla}_X\alpha)(Y,Z,\eta) = \left(\dfrac{a''}{a}-\dfrac{(a')^2}{a^2}+\dfrac{\varepsilon c}{a^2}\right)\left<\xi,\eta\right>\Big.\left(\left<T,X\right>\left<Y,Z\right>-\left<T,Y\right>\left<X,Z\right>\Big.\right),
\]\\
for any $X,Y,Z\in TM$ and $\eta\in TM^{\perp}.$
\begin{proof}
The classical Codazzi equation for submanifolds states
\[
(\bar{\nabla}_Y\alpha)(X,Z,\eta) - (\bar{\nabla}_X\alpha)(Y,Z,\eta) = - \bar{R}(X,Y,Z,\eta). 
\]
The Proposition \ref{curvature tensor bar{R}} gives us
\[
\bar{R}(X,Y,Z,\eta) = \left(\dfrac{a''}{a}-\dfrac{(a')^2}{a^2}+\dfrac{\varepsilon c}{a^2}\right)\left<\xi,\eta\right>\Big.\left(\left<T,Y\right>\left<X,Z\right>-\left<T,X\right>\left<Y,Z\right>\Big.\right),
\]
and the proof is complete. \hfill $\square$
\end{proof}

\end{proposition}

\begin{proposition} \label{Ricci for a given immersion}
The Ricci equation of $M$ in $\varepsilon I\times_a\mathbb{M}^N_{\lambda}(c)$ is
$$
R^{\perp}(X,Y)\eta = \alpha(A_\eta Y,X)-\alpha(A_\eta X,Y),
$$
for any $X,Y\in TM$ and $\eta\in TM^{\perp}.$
\begin{proof}
The classical Ricci equation for submanifolds states
\[
R^{\perp}(X,Y)\eta = \bar{R}(X,Y)\eta + \alpha(A_\eta Y,X)-\alpha(A_\eta X,Y).
\] 
The Proposition \ref{curvature tensor bar{R}} gives us $\bar{R}(X,Y)\eta=0$, and the proof is complete. \hfill $\square$
\end{proof}
\end{proposition}

\section{Structure equations and main theorem} \label{MT}
In the previous section we found necessary conditions for a nondegenerate arbitrary signature manifold $M^n$ to be realized as a submanifold in the large class of warped product manifolds $\varepsilon I\times_a\mathbb{M}^{N}_{\lambda}(c)$, where $\varepsilon=\pm 1,\ a:I\subset\mathbb{R}\to\mathbb{R}^+$ is the scale factor and $\mathbb{M}^{N}_{\lambda}(c)$ is the $N$-dimensional semi-Riemannian space form of index $\lambda$ and constant curvature $c\in\{-1,1\}.$ Now, we want to prove the converse. Namely, we want to show that if $M^n$ satisfies Gauss, Codazzi and Ricci equations for a submanifold in $\varepsilon I\times_a\mathbb{M}^{N}_{\lambda}(c)$, along with some additional conditions that appear in Lemma \ref{derivative rules in M immersed}, then $M^n$ can be isometrically immersed into $\varepsilon I\times_a\mathbb{M}^{N}_{\lambda}(c)$.\\

Élie Cartan developed the moving frame technique. Definitions, basic results and some other details can be found in \cite{IL}. We will use the following convention on the ranges of indices, unless mentioned otherwise:
\[
0\leq\al,\be,\ga,\cdots \leq N+1; \hspace{1.3cm} 1\leq i,j,k,\cdots\leq n; \hspace{1.3cm} n+1 \leq u,v,w\cdots\leq N+1
\]

We consider $\tilde{P}^{N+2}=\varepsilon I\times_a\mathbb{E}_{\lambda}^{N+1}$ with a connection $\tilde{\nabla}$ and a base $\lp E_0,\ldots,E_{N},\partial_t\rp$, where $\lp E_0,\ldots,E_{N}\rp$ is an orthonormal frame of $\mathbb{E}_{\lambda}^{N+1}$. We set $ \overline{E}_0=\dfrac{E_0}{ca},\ldots, \overline{E}_N=\dfrac{E_N}{ca}, \overline{E}_{N+1}=\partial_t$, and, consequently, $\lp\overline{E}_0,\ldots,\overline{E}_{N+1}\rp$ is an orthonormal base of $\varepsilon I\times_a\mathbb{E}_{\lambda}^{N+1}$. If necessary, we reorder $\lp\overline{E}_1,\ldots,\overline{E}_{N+1}\rp$ so that $\pe \overline{E}_\al,\overline{E}_\al\pd_{\varepsilon I\times_a\mathbb{E}_{\lambda}^{N+1}}=\varepsilon_\al=\pm 1$ and suppose, without lost of generality, that $\overline{E}_{N+1}=\partial_t$ and $c=\varepsilon_0$. Also, let $\overline{\nabla}$ be the Levi-Civita connection of $\bar{P}^{N+1}=\varepsilon I\times_a\mathbb{M}_{\lambda}^{N}(c)$.\\

From now on, let $(M^n,\langle,\rangle_M)$ be a semi-Riemannian manifold of index $p$ and $(E,\langle,\rangle_E)$ a semi-Riemannian vector bundle of index $q$ and rank $m=N+1-n$ over $M$ with compatible connection  $\nabla^E$ and curvature operator $R^E$. Let us also give $\alpha^E$ a simmetric section in $\Gamma(T^*M\otimes T^*M\otimes E),\ \xi$ a section in $\Gamma(E)$, real numbers $c,\varepsilon\in\{-1,1\}$ and smooth functions $a:I\subset\mathbb{R}\rightarrow\mathbb{R}^+$ and $\pi:M\rightarrow I$. We define the vector field $T\in TM$ by $T=\varepsilon\cdot \text{grad}(\pi)$ and, for each $\eta\in\Gamma(E)$, we define the section $A_\eta\in\Gamma(T^* M\otimes TM)$ by $\left<\alpha^{E}(X,Y),\eta\right>=\left<A_\eta(X),Y\right>$.

\begin{definition} \label{structure equations}
	
	Under the previous conditions, we say that $M$ satisfies the \textbf{structure equations} if the following conditions hold for any $X,Y,Z,W\in TM$ and $\eta\in\Gamma(E)$:

\begin{itemize}
	
\item[(A)] $\left<T,T\right>+\left<\xi,\xi\right>=\varepsilon.$
\item[(B)] $\nabla_X T=\dfrac{a'}{a}\left(X-\varepsilon\left<X,T\right>T \right)+A_{\xi}X.$
\item[(C)] $\nabla_X^E\xi=\dfrac{-\varepsilon (a'}{a}\left<X,T\right>\xi-\alpha^E(T,X).$
\item[(D)] (Gauss) 
\begin{eqnarray*}
R(X,Y,Z,W) &=& \Bigg(\varepsilon\dfrac{(a')^2}{a^2}-\dfrac{c}{a^2}\Bigg)\Big(\left<X,Z\right>\left<Y,W\right>-\left<Y,Z\right>\left<X,W\right>\Big)\\
           && +\left(\dfrac{a''}{a}-\dfrac{(a')^2}{a^2}+\dfrac{\varepsilon c}{a^2}\right)\Big(\left<X,Z\right>\left<Y,T\right>\left<W,T\right>-\left<Y,Z\right>\left<X,T\right>\left<W,T\right>\Big.\\
           &&-\Big.\left<X,W\right>\left<Y,T\right>\left<Z,T\right>+\left<Y,W\right>\left<X,T\right>\left<Z,T\right>\Big)\\
           &&-\left<\alpha(X,Z),\alpha(Y,W)\right>+\left<\alpha(X,W),\alpha(Y,Z)\right>.\\
\end{eqnarray*}
\item[(E)] (Codazzi)
\[
(\nabla_Y\alpha^E)(X,Z,\eta) - (\nabla_X\alpha^E)(Y,Z,\eta) = \left(\dfrac{a''}{a}-\dfrac{(a')^2}{a^2}+\dfrac{\varepsilon c}{a^2}\right)\left<\xi,\eta\right>\Big.\left(\left<T,X\right>\left<Y,Z\right>-\left<T,Y\right>\left<X,Z\right>\Big.\right).
\]
\item[(F)] (Ricci)
$$
R^E(X,Y)\eta = \alpha^E(A_\eta Y,X)-\alpha^E(A_\eta X,Y).
$$
\end{itemize}
\end{definition}	
\begin{remark}
By abuse of notation, we have written above $a=a\circ\pi.$ 
\end{remark}

\begin{theorem} \label{main theorem}
\begin{itemize}
\item[(i)] {\bf Existence:} Assume that $M$, under the previous conditions, is simply connected and satisfies the structure equations. Then there exists an isometric immersion $f:M\rightarrow \varepsilon I\times_a \mathbb{M}_{\lambda}^N(c)$  and a vector bundle isometry $\Phi:E\rightarrow Tf(M)^{\perp}$, such that: 

\begin{itemize}
	\item[1.] $\partial_t=df(T)+\Phi(\xi).$
	\item[2.] $\pi=\pi_I\circ f$, where $\pi_I:\varepsilon I\times_a \mathbb{M}_{\lambda}^N(c)\rightarrow I$ is the projection.
	\item[3.] $\alpha_f=\Phi\circ\alpha^E\circ df^{-1}$, where $\al_f$ is the second fundamental form of immersion $f$.
	\item[4.] (D), (E) e (F) are the Gauss, Codazzi and Ricci equations, respectily, of the immersion founded.
	\item[5.] $\nabla^{\perp}\Phi=\Phi\nabla^E.$
\end{itemize}

\item[(ii)] {\bf Uniqueness:} Let $f,g : M \to\varepsilon I\times_a \mathbb{M}_{\lambda}^N(c)$ be isometric immersions of a semi-Riemannian manifold. Assume that there exists a vector bundle isometry 
$\Phi:T_fM^{\perp}\to T_gM^{\perp}$ such that
\[
\Phi\alpha_f = \alpha_g \ \ \ \textrm{and} \ \ \ \Phi^f\nabla^{\perp}={}^g\nabla^{\perp}\Phi.
\]
Then there exists an isometry $\tau:\varepsilon I\times_a \mathbb{M}_{\lambda}^N(c)\to \varepsilon I\times_a \mathbb{M}_{\lambda}^N(c)$ such that $\tau\circ f = g$ and $\tau_*|_{T_fM^{\perp}}=\Phi.$
\end{itemize}
\end{theorem}	

\noindent Regarding the part $(i)$, let $\nabla$ be the Levi-Civita connection on $TM$. Consider the Whitney sum $F=TM\oplus E$ endowed with the orthogonal sum of the metrics in $TM$ and $E$. Define

\begin{eqnarray*}
\nabla^{\atop F}_XY& = &\nabla_XY+\alpha^{E}(X,Y),\hspace{0.2cm} X,Y\in\Gamma(TM),\\
\nabla^{\atop F}_X\eta& = &-A_\eta X+\nabla^{\atop E}_X\eta,\hspace{0.2cm} X\in\Gamma(TM),\hspace{0.2cm} \eta\in\Gamma(E).
\end{eqnarray*}

\noindent It is easy to see that $\nabla^{ F}$ is a compatible connection on $F$. Let us give a point $x\in M$ and consider around it a local orthonormal frame  $\{e_1,\ldots,e_n,e_{n+1},\ldots,e_{n+m}\}$ for $F$, with theirs signs $\varepsilon_{\alpha}=\left<e_{\alpha},e_{\alpha}\right>=\pm1$, where $\{e_{1},\ldots e_{n}\}$ is a local orthonormal frame for $TM$ and $\{e_{n+1},\ldots e_{n+m}\}$ is a local orthonormal frame for $E$. Define in $\Gamma(F^*)$ the 1-forms\\

$$ 
\omega_{\alpha}=e^*_{\alpha},\text{ se } \alpha\in\{1,\ldots,n+m\},\text{ e }\omega_0=0. 
$$\\
\noindent Also define $\delta\in \Gamma(F^*)$ by $\delta(X)=\left<X,T+\xi\right>$, for $X\in F$. With the help of the tensor $S$ in $\Gamma(TM^*\otimes F)$ given by $SX=\dfrac{-1}{ac}\left(X-\varepsilon \left<X,T+\xi\right>(T+\xi)\right)$, we construct the following $1$-forms

$$
\omega_{\alpha0}(X)=-\varepsilon_\alpha\left<e_\alpha,SX\right>, \hspace{0.3cm} \omega_{ij}(X)=\varepsilon_i\left<e_i,\nabla^{\atop F}_Xe_j\right>=\varepsilon_i\left<e_i,\nabla_Xe_j\right>,
$$
$$
\omega_{iu}(X)=\varepsilon_i\left<e_i,\nabla^{\atop F}_Xe_u\right>=-\varepsilon_i\left<e_i,A_{e_u}(X)\right>, \hspace{0.3cm} \omega_{uv}(X)=\varepsilon_u\left<e_u,\nabla^{\atop F}_Xe_v\right>=\varepsilon_u\left<e_u,\nabla_Xe_v\right>, 
$$

$$
\omega_{\alpha\beta}=-\varepsilon_{\alpha}\varepsilon_{\beta}\omega_{\beta\alpha},
$$
\\
\noindent for any $X\in TM$, which are known as the connection 1-forms. Now, define the functions $T_\alpha=\delta(e_\alpha)$ for $\alpha\in\{1,\ldots,n+m\}$ and $T_0=0$. Next, we consider the matrices $\textbf{X}$ e $\Upsilon$, given by

$$
\textbf{X}_{\alpha\beta}= \dfrac{\varepsilon a'}{a}\Big\{T_{\beta}\omega_{\alpha}-\varepsilon_{\alpha}\varepsilon_{\beta}T_{\alpha}\omega_{\beta}\Big\},\hspace{0.3cm}\Upsilon=\Omega-\textbf{X},
$$
where $\Omega = (\omega_{\alpha\beta}).$
\noindent A simple computation gives

\begin{equation}\label{eq1}
d\Upsilon+\Upsilon\wedge\Upsilon=d\Omega-d\textbf{X}+\Omega\wedge\Omega-\Omega\wedge\textbf{X}-\textbf{X}\wedge\Omega+\textbf{X}\wedge\x.
\end{equation}

\noindent Our goal now is to prove that the right side of (\ref{eq1}) vanishes, and the proof of it will be split in some lemmas.

\begin{lemma}\label{aux}

The following equalities hold:

\begin{itemize}
	\item[1.] $\displaystyle\sum_{\gamma}\varepsilon_\gamma T_\gamma^2=\varepsilon;$
	\item[2.] $\delta=\displaystyle\sum_{\gamma}T_\gamma \omega_\gamma;$
	\item[3.] $dT_\alpha= \displaystyle\sum_{\gamma}T_\gamma \omega_{\gamma\alpha}+\dfrac{a'}{a}\varepsilon_\alpha \omega_\alpha-\dfrac{\varepsilon a'}{a}T_\alpha \delta$;
	\item[4.] $d\cal{W}=$$-\Omega\wedge\cal{W}$, where $\cal{W}$$=(\w_{0},\ldots,\w_{n+m})^\top.$
\end{itemize}

\end{lemma}
\begin{proof}
Item (1) comes from item (A) of the {\ce}, and (2) is immediate from definitions of the  $\delta, T_\gamma$ and $\w_\gamma$. In order to prove item (3), we see that if $X\in TM$ and $\alpha\geq1$, then

\begin{eqnarray*}
\ds\sum_{\ga}T_\gamma \omega_{\gamma\alpha}(X) & = & \ds\sum_{\gamma}\left<\eg,T+\xi\right>\vg\pe\eg,\nabla^{\atop F}_X\ea\pd \\
                                         							           & = & \pe T+\xi,\nabla^{\atop F}_X\ea\pd \\
                                         							           & = & X\left(\pe T+\xi,\ea\pd \right)-\pe\nabla^{\atop F}_X(T+\xi),\ea\pd \\
                                         							           & = & X(T_\alpha)-\pe \dfrac{a'}{a}\left(X-\varepsilon\left<X,T\right>(T+\xi) \right),\ea\pd \\
                                         							           & = & dT_\alpha(X) -\dfrac{a'}{a}\varepsilon\w_\alpha(X)+\dfrac{\varepsilon a'}{a}T_\alpha \delta(X),
\end{eqnarray*}

\noindent where the penultimate equality was obtained from the itens (B) and (C) of the {\ce} and the definition of the connection $\nabla^F$. If $\al=0$, 

\begin{eqnarray*}
\ds\sum_{\ga}T_\gamma \omega_{\gamma0}(X) & = &- \ds\sum_{\gamma\geq1}\left<\eg,T+\xi\right>\vg\pe\eg,SX\pd \\
                                         							           & = & -\pe T+\xi,SX\pd \\
                                         							           & = & 0.
\end{eqnarray*}

\noindent Then, $dT_0= \displaystyle\sum_{\gamma}T_\gamma \omega_{\gamma0}+\dfrac{a'}{a}\varepsilon_0 \omega_0-\dfrac{\varepsilon a'}{a}T_0 \delta$, because $dT_0=T_0=\w_0=0$.

\noindent Finally, to prove item (4), we see that if $\alpha\geq1$, then

\begin{eqnarray*}
d\w_{\alpha}(X,Y) & = & X(\w_{\alpha}(Y))-Y(\w_\alpha (X))-\w_{\alpha}([X,Y])\\
		  	    & = & X(\va\pe\ea,Y\pd)-Y(\va\pe\ea,X\pd)-\va\pe\ea,[X,Y]\pd)\\
			    & = & \va\left(\pe\nf_X\ea,Y\pd+\pe\ea,\nf_XY\pd\right)-\va\left(\pe\nf_Y\ea,X\pd+\pe\ea,\nf_YX\pd\right)-\va\pe\ea,[X,Y]\pd)\\
			   & = & \va\pe\nf_X\ea,Y\pd-\va\pe\nf_Y\ea,X\pd\\ 
			  & = &\va\left(\ds\sum_{\gamma\geq1}\omega_{\gamma\alpha}(X)\vg\w_{\gamma}(Y)-\omega_{\gamma\alpha}(Y)\vg\w_{\gamma}(X)\right)\\
			& = &\va\left(\ds\sum_{\gamma\geq0}\omega_{\gamma\alpha}(X)\vg\w_{\gamma}(Y)-\omega_{\gamma\alpha}(Y)\vg\w_{\gamma}(X)\right)\\	 
			& = & -\ds\sum_{\gamma\geq0}\w_{\alpha\gamma}\wedge\w_\gamma(X,Y)\\
			& = & -(\Omega\wedge{\cal{W}})_\alpha(X,Y).
\end{eqnarray*}

\noindent If $\alpha=0$, then $d\w_\alpha=0$, and

\begin{eqnarray*}
(\Omega\wedge{\cal{W}})_0(X,Y)&=&\ds\sum_{\gamma}\omega_{0\gamma}(X)\w_{\gamma}(Y)-\omega_{0\gamma}(Y)\w_{\gamma}(X)\\
&=& \dfrac{\varepsilon}{a}\delta\wedge\left(\ds\sum_\gamma T_\gamma\w_\gamma\right)(X,Y)\\
&=& \dfrac{\varepsilon}{a}\delta\wedge\delta(X,Y)=0.
\end{eqnarray*}
This concludes the proof.\hfill $\square$
\end{proof}

\begin{lemma}\label{lema_im}

The following equalities hold:

\begin{itemize}
	\item[1.] $(d\x)_{\alpha\beta}=\varepsilon\left( \dfrac{aa''-(a')^2}{aa'}\right)\delta\wedge\x_{\alpha\beta}+\dfrac{\varepsilon a'}{a}\left(dT_\beta\wedge\w_\alpha-\va\vb dT_\alpha\wedge\w_\beta\right)+\dfrac{\varepsilon a'}{a}\left(T_\beta d\w_\alpha-\va\vb T_\alpha d\w_\beta\right)$;
	\item[2.] $(\x\wedge \x)_{\alpha\beta}=-\left(\dfrac{ \varepsilon a'}{a}\right)\delta\wedge\x_{\alpha\beta}-\left(\dfrac{ a'}{a}\right)^2\varepsilon\vb\w_\alpha\wedge\w_\beta$;
	\item[3.] $(\Omega\wedge \x)_{\alpha\beta}+(\x\wedge \Omega)_{\alpha\beta}=-\left(\dfrac{ \varepsilon a'}{a}\right)\lp T_\be d\wa-T_\al\va\vb d\wb\rp-\left(\dfrac{ \varepsilon a'}{a}\right)\lp dT_\be\we\wa-\va\vb dT_\al \we\wb\rp\\
	   -\left(\dfrac{ \varepsilon a'}{a}\right)\delta\we\x_{\al\be}-2\lp\dfrac{a'}{a}\rp^2 \varepsilon\vb\wa\we\omega_\be$;
	\item[4.] $(d\Omega)_{\alpha\beta}+(\Omega\wedge\Omega)_{\alpha\beta}=-\left(\dfrac{a'}{a}\right)^2\varepsilon\vb\wa\we\wb + \varepsilon\left( \dfrac{aa''-(a')^2}{aa'}\right)\delta\wedge\x_{\alpha\beta}.$
\end{itemize}
\end{lemma}

\begin{proof}

First of all, it is important to remember that $T=\varepsilon\cdot\textrm{grad}(\pi)$ and that we have been written $a=a\circ\pi$. For (1), we have

\begin{eqnarray*}
(d\x)_{\alpha\beta}&=&d\left(\dfrac{\varepsilon a'}{a}\Big\{T_{\beta}\omega_{\alpha}-\varepsilon_{\alpha}\varepsilon_{\beta}T_{\alpha}\omega_{\beta}\Big\}\right)\\
&=& d\left(\dfrac{\varepsilon a'}{a}\right)\wedge\Big\{T_{\beta}\omega_{\alpha}-\varepsilon_{\alpha}\varepsilon_{\beta}T_{\alpha}\omega_{\beta}\Big\}+\dfrac{\varepsilon a'}{a}d\left(\Big\{T_{\beta}\omega_{\alpha}-\varepsilon_{\alpha}\varepsilon_{\beta}T_{\alpha}\omega_{\beta}\Big\}\right)\\
&=&\varepsilon\left( \dfrac{aa''-(a')^2}{a^2}\right)\dfrac{ a}{a'}\delta\wedge\x_{\alpha\beta}+\dfrac{\varepsilon a'}{a}\left(dT_\beta\wedge\w_\alpha-\va\vb dT_\alpha\wedge\w_\beta\right)+\dfrac{\varepsilon a'}{a}\left(T_\beta d\w_\alpha-\va\vb T_\alpha d\w_\beta\right).
\end{eqnarray*}

\noindent For (2), we have
\begin{eqnarray*}
(\x\wedge\x)_{\alpha\beta}&=&\ds\sum_\gamma\x_{\alpha\gamma}\wedge\x_{\gamma\beta}\\
&=& \left(\dfrac{ a'}{a}\right)^2\ds\sum_\gamma\Big(T_{\gamma}\omega_{\alpha}-\varepsilon_{\alpha}\varepsilon_{\gamma}T_{\alpha}\omega_{\gamma}\Big)\wedge\Big(T_{\beta}\omega_{\gamma}-\varepsilon_{\gamma}\varepsilon_{\beta}T_{\gamma}\omega_{\beta}\Big)\\
&=&\left(\dfrac{ a'}{a}\right)^2\left(T_\beta\w_\alpha\wedge\ds\sum_\gamma T_\gamma\w_\gamma\right)-\left(\dfrac{ a'}{a}\right)^2\vb\left(\ds\sum_\gamma \vg T^2_\gamma\right)\w_\alpha\wedge\w_\beta\\
&&+\left(\dfrac{ a'}{a}\right)^2\va\vb T_\alpha\left(\ds\sum_{\gamma}T_\gamma\w_\gamma\right)\wedge\w_\beta\\
&=&\left(\dfrac{ a'}{a}\right)^2\left(T_\beta\w_\alpha\wedge\delta\right)-\left(\dfrac{ a'}{a}\right)^2\vb\varepsilon\w_\alpha\wedge\w_\beta+\left(\dfrac{ a'}{a}\right)^2\va\vb T_\alpha\delta\wedge\w_\beta\\
&=&\left(\dfrac{ \varepsilon a'}{a}\right)\x_{\alpha\beta}\wedge\delta-\left(\dfrac{ a'}{a}\right)^2\varepsilon\vb\w_\alpha\wedge\w_\beta.
\end{eqnarray*}

\noindent For (3), we initially calculate $(\Omega\wedge \x)_{\alpha\beta}$. We have

\begin{eqnarray*}
(\Omega\wedge \x)_{\alpha\beta}&=&\ds\sum_{\ga}\w_{\alpha\ga}\we\x_{\ga\be}\\
&=& \left(\dfrac{ \varepsilon a'}{a}\right)\ds\sum_\ga\w_{\al\ga}\we\Big(T_{\beta}\omega_{\gamma}-\varepsilon_{\gamma}\varepsilon_{\beta}T_{\gamma}\omega_{\beta}\Big)\\
&=& \left(\dfrac{ \varepsilon a'}{a}\right)T_{\beta}\ds\sum_\ga\w_{\al\ga}\we \omega_{\gamma}+\left(\dfrac{ \varepsilon a'}{a}\right)\va\varepsilon_{\beta}\left(\ds\sum_\ga T_{\gamma}\w_{\ga\al}\right)\we\omega_{\beta}\\
&=& \left(\dfrac{ \varepsilon a'}{a}\right)T_{\beta}\ds\sum_\ga\w_{\al\ga}\we \omega_{\gamma}+\left(\dfrac{ \varepsilon a'}{a}\right)\va\varepsilon_{\beta}\left(dT_\alpha-\dfrac{a'}{a}\varepsilon_\alpha \omega_\alpha+\dfrac{\varepsilon a'}{a}T_\alpha \delta\right)\we\omega_{\beta}\\
&=&- \left(\dfrac{ \varepsilon a'}{a}\right)T_{\beta}d\wa+\left(\dfrac{ \varepsilon a'}{a}\right)\va\varepsilon_{\beta}dT_\alpha\we\wb-\lp\dfrac{a'}{a}\rp^2\varepsilon\vb \omega_\alpha\we\wb+\lp\dfrac{a'}{a}\rp^2\va\vb T_\alpha \delta\we\omega_{\beta}.
\end{eqnarray*}

Now, we compute compute $( \x\wedge\Omega)_{\alpha\beta}$. We have

\begin{eqnarray*}
(\x\we\Omega)_{\alpha\beta}&=&\ds\sum_{\ga}\x_{\al\ga}\we\w_{\ga\be}\\
&=& \left(\dfrac{ \varepsilon a'}{a}\right)\ds\sum_\ga\Big(T_{\ga}\omega_{\al}-\varepsilon_{\al}\varepsilon_{\gamma}T_{\al}\omega_{\ga}\Big)\we \w_{\ga\be}\\
&=& \left(\dfrac{ \varepsilon a'}{a}\right)\wa\we\lp\ds\sum_\ga T_{\ga}\w_{\ga\be}\rp-\left(\dfrac{ \varepsilon a'}{a}\right)T_\al\varepsilon_{\al}\vb\left(\ds\sum_\ga \w_{\be\ga}\we\wg\right)\\
&=& \left(\dfrac{ \varepsilon a'}{a}\right)\wa\we\lp dT_\be-\dfrac{a'}{a}\varepsilon_\be \omega_\be+\dfrac{\varepsilon a'}{a}T_\be \delta\rp-\left(\dfrac{ \varepsilon a'}{a}\right)T_\al\varepsilon_{\al}\vb\left(\ds\sum_\ga \w_{\be\ga}\we\wg\right)\\
&=& -\left(\dfrac{ \varepsilon a'}{a}\right) dT_\be\we\wa-\lp\dfrac{a'}{a}\rp^2 \varepsilon\vb \wa\we\omega_\be-\lp\dfrac{a'}{a}\rp^2T_\be\delta\we\wa+\left(\dfrac{ \varepsilon a'}{a}\right)T_\al\varepsilon_{\al}\vb d\wb.
\end{eqnarray*}

\noindent Hence, 

$
(\Omega\wedge \x)_{\alpha\beta}+(\x\wedge \Omega)_{\alpha\beta}=-\left(\dfrac{ \varepsilon a'}{a}\right)\lp T_\be d\wa-T_\al\va\vb d\wb\rp-\left(\dfrac{ \varepsilon a'}{a}\right)\lp dT_\be\we\wa-\va\vb dT_\al \we\wb\rp-\left(\dfrac{ \varepsilon a'}{a}\right)\delta\we\x_{\al\be}-2\lp\dfrac{a'}{a}\rp^2 \varepsilon\vb\wa\we\omega_\be.
$\\

\noindent For (4), if $\al,\be\geq1$, a straightforward computation gives
$$
(d\Omega)_{\al\be}+(\Omega\we\Omega)_{\al\be}=\va\pe R^{F}(X,Y)\eb,\ea\pd-\w_{\al0}\we\w_{0\be}(X,Y).
$$

First, we calculate $\pe R^{F}(X,Y)\eb,\ea\pd$ in some cases:\\

Case 1: $\al,\be\in\{1,\ldots,n\}$

\begin{eqnarray*}
\pe R^{F}(X,Y)\eb,\ea\pd&=&\pe\nf_X\nf_Y\eb-\nf_Y\nf_X\eb-\nf_{[X,Y]}\eb,\ea\pd\\
&=&\pe\nf_X\lp\nabla_Y\eb+\al^E(\eb,Y)\rp-\nf_Y\lp\nabla_X\eb+\al^E(\eb,X)\rp-\lp\nabla_{[X,Y]}\eb+\al^E\lp\eb,[X,Y]\rp\rp,\ea\pd\\
&=&\pe\nf_X\nabla_Y\eb+\nf_X\al^E(\eb,Y)-\nf_Y\nabla_X\eb-\nf_Y\al^E(\eb,X)-\nabla_{[X,Y]}\eb-\al^E\lp\eb,[X,Y]\rp,\ea\pd\\
&=&\pe\nabla_X\nabla_Y\eb,\ea\pd+\pe\al^E\lp X,\nabla_Y\eb\rp,\ea\pd-\pe A_{\al^E\lp\eb,Y\rp}X,\ea\pd+\pe\nabla^E_X\al^E(\eb,Y),\ea\pd\\
&-&\pe\nabla_Y\nabla_X\eb,\ea\pd-\pe\al^E\lp Y,\nabla_X\eb\rp,\ea\pd+\pe A_{\al^E\lp\eb,X\rp}Y,\ea\pd-\pe\nabla^E_Y\al^E(\eb,X),\ea\pd\\
&-&\pe\nabla_{[X,Y]}\eb,\ea\pd-\pe\al^E\lp\eb,[X,Y]\rp,\ea\pd\\
&=&\pe R(X,Y)\eb,\ea\pd+\pe \lp\nabla_X\al\rp(\eb,Y)-\lp\nabla_X\al\rp(\eb,Y),\ea\pd-\pe A_{\al^E\lp\eb,Y\rp}X-A_{\al^E\lp\eb,X\rp}Y,\ea\pd.
\end{eqnarray*}

\noindent Since $\ea\in TM$, 
\begin{eqnarray*}
\pe R^{F}(X,Y)\eb,\ea\pd&=&\pe R(X,Y)\eb,\ea\pd-\pe \al^E\lp\eb,Y\rp, \al^E\lp\ea,X\rp\pd+\pe \al^E\lp\eb,X\rp, \al^E\lp\ea,Y\rp\pd\\
&=&\Bigg(\varepsilon\dfrac{(a')^2}{a^2}-\dfrac{c}{a^2}\Bigg)\Big(\left<X,\eb\right>\left<Y,\ea\right>-\left<Y,\eb\right>\left<X,\ea\right>\Big)\\
&& +\left(\dfrac{a''}{a}-\dfrac{(a')^2}{a^2}+\dfrac{\varepsilon c}{a^2}\right)\Big(\left<X,\eb\right>\delta(Y)T_\al-\left<Y,\eb\right>\delta(X)T_\al\Big.
-\Big.\left<X,\ea\right>\delta(Y)T_\be+\left<Y,\ea\right>\delta(X)T_\be\Big)\\
&=&-\Bigg(\varepsilon\dfrac{(a')^2}{a^2}-\dfrac{c}{a^2}\Bigg)\va\vb\wa\we\wb(X,Y)\\
&& -\left(\dfrac{a''}{a}-\dfrac{(a')^2}{a^2}+\dfrac{\varepsilon c}{a^2}\right)\va(T_\be\wa-T_\al\va\vb\wa)\we\delta(X,Y)\\
&=&-\Bigg(\varepsilon\dfrac{(a')^2}{a^2}-\dfrac{c}{a^2}\Bigg)\va\vb\wa\we\wb(X,Y)-\left(\dfrac{a''}{a}-\dfrac{(a')^2}{a^2}+\dfrac{\varepsilon c}{a^2}\right)\dfrac{a}{a'}\va\varepsilon\x_{\al\be}\we\delta(X,Y).
\end{eqnarray*}

\noindent But, it is easy to see that

$$
\w_{\al0}\we\w_{0\be}=\dfrac{-\varepsilon_0\vb}{a^2}\wa\we\wb+\dfrac{\varepsilon_0\varepsilon}{aa'}\x_{\al\be}\we\delta,
$$

\noindent and, in this case, we conclude that

$$
(d\Omega)_{\al\be}+(\Omega\we\Omega)_{\al\be}=-\varepsilon\dfrac{(a')^2}{a^2}\vb\wa\we\wb-\left(\dfrac{a''}{a}-\dfrac{(a')^2}{a^2}\right)\dfrac{\varepsilon a}{a'}\x_{\al\be}\we\delta.
$$
\\
\noindent For the other cases, we use a similar computation, noting that $\w_\al(TM)=0$ if $\al\in\{n+1,\ldots,n+m\}$. Now, for $\al$ or $\be$ equal zero, we have, considering $\beta = 0$ for instance,

\begin{eqnarray*}
(d\Omega)_{\al0}+(\Omega\we\Omega)_{\al0}&=&-\va X\pe\ea,SY\pd+\va Y\pe\ea,SX\pd+\va \pe\ea,S[X,Y]\pd\\
&+&\ds\sum_{\ga\geq1}-\va\vg\pe\ea,\nf_X\eg\pd\pe\eg,SY\pd-\ds\sum_{\ga\geq1}-\va\vg\pe\ea,\nf_Y\eg\pd\pe\eg,SX\pd\\
&=&-\va X\pe\ea,SY\pd+\va Y\pe\ea,SX\pd+\va \pe\ea,S[X,Y]\pd\\
&+&\va\ds\sum_{\ga\geq1}\vg\pe\nf_X\ea,\eg\pd\pe\eg,SY\pd-\va\ds\sum_{\ga\geq1}\vg\pe\nf_Y\ea,\eg\pd\pe\eg,SX\pd\\
&=&-\va X\pe\ea,SY\pd+\va Y\pe\ea,SX\pd+\va \pe\ea,S[X,Y]\pd+\va\pe\nf_X\ea,SY\pd-\va\pe\nf_Y\ea,SX\pd\\
&=&-\va\pe\ea,\nf_X SY\pd+\va\pe\ea,\nf_Y SX\pd+\va \pe\ea,S[X,Y]\pd\\
&=&\va\pe\ea,-\nf_X SY+\nf_Y SX+S[X,Y]\pd.
\end{eqnarray*}

\noindent Now, we need to compute $-\nf_X SY+\nf_Y SX+S[X,Y]$. First, we note that for any $U\in \Gamma(F)$ it holds $\nf_X\lp\dfrac{-1}{ac}U\rp=\dfrac{\varepsilon a'}{ca^2}\delta(X)U-\dfrac{1}{a\varepsilon_0}\nf_XU$ and $\nf_X(T+\xi)=-ca'SX$. Therefore,

\begin{eqnarray*}
\nf_XSY&=&\nf_X\lp\dfrac{-1}{ac}\lp Y-\varepsilon\delta(Y)(T+\xi)\rp\rp\\
&=&\dfrac{\varepsilon a'}{ca^2}\delta(X)\lp Y-\varepsilon\delta(Y)(T+\xi)\rp-\dfrac{1}{ac}\nf_X\lp Y-\varepsilon\delta(Y)(T+\xi)\rp\\
&=&\dfrac{\varepsilon a'}{ca^2}\delta(X) Y-\dfrac{a'}{ca^2}\delta(X)\delta(Y)(T+\xi)-\dfrac{1}{ac}\nf_X Y+\dfrac{\varepsilon}{ac}X(\delta(Y))(T+\xi)-\dfrac{\varepsilon a'}{a}\delta(Y)SX.
\end{eqnarray*} 

\noindent Analogously,

$$
\nf_YSX=\dfrac{\varepsilon a'}{ca^2}\delta(Y) X-\dfrac{a'}{c}\delta(Y)\delta(X)(T+\xi)-\dfrac{1}{ac}\nf_Y X+\dfrac{\varepsilon}{ac}Y(\delta(X))(T+\xi)-\dfrac{\varepsilon a'}{a}\delta(X)SY.
$$

\noindent Then,

\begin{eqnarray*}
\nf_YSX-\nf_XSY+S[X,Y]&=&\dfrac{\varepsilon a'}{ca^2}(\delta(Y) X-\delta(X) Y)+\dfrac{1}{ac}[X,Y]+\dfrac{\varepsilon}{ac}(Y(\delta(X))-X(\delta(Y))(T+\xi)\\
&& -\dfrac{\varepsilon a'}{a}(\delta(X)SY-\delta(Y)SX)+S[X,Y]\\
&=&\dfrac{\varepsilon a'}{ca^2}(\delta(Y) X-\delta(X) Y)+\dfrac{1}{ac}[X,Y]-\dfrac{\varepsilon}{ac}\delta([X,Y])(T+\xi)\\
&& -\dfrac{\varepsilon a'}{a}(\delta(X)SY-\delta(Y)SX)+S[X,Y]\\
&=&\dfrac{\varepsilon a'}{ca^2}(\delta(Y) X-\delta(X) Y)-\dfrac{\varepsilon a'}{a}(\delta(X)SY-\delta(Y)SX)\\
&=&0.
\end{eqnarray*}

\noindent Therefore, it follows that $(d\Omega)_{\al0}+(\Omega\we\Omega)_{\al0}=0.$ Since 
\[
-\varepsilon\dfrac{(a')^2}{a^2}\vb\wa\we\wb-\left(\dfrac{a''}{a}-\dfrac{(a')^2}{a^2}\right)\dfrac{\varepsilon a}{a'}\x_{\al\be}\we\delta
\]
vanishes when $\be=0$, the item (4) is proved. This concludes the proof. \hfill $\square$
\end{proof}

\begin{lemma}
$d\Upsilon+\Upsilon\we\Upsilon=0$.
\end{lemma}

\begin{proof}
\noindent The itens of Lemma \ref{lema_im} say that the right side of the equation (\ref{eq1}) vanishes, and this finishes the proof. \hfill $\square$
\end{proof}

\noindent Set $N=m+n-1$, $\lambda=p+q+\dfrac{|c-1|}{2}$  and define $\mathcal{S}=\{Z\in\mathcal{M}_{N+2}(\mathbb{R});Z^tGZ=G\}$, where $G_{\al\be}=\va\delta_{\al\be}.$ Also set the map $s:\mathcal{S}\rightarrow\mathbb{S}\lp \mathbb{E}^{\atop N+2}_{\lambda}\rp=\{X\in\mathbb{E}^{\atop N+2}_{\lambda};\pe X,X\pd=\varepsilon\}$ given by $Z\mapsto\lp Z_{(N+1) 0},\ldots,Z_{N+1 N+1}\rp^t$.

\begin{proposition} \label{submersion s}
	
	The map $s$ described above is a submersion.

\end{proposition}

\begin{proof}
	
	\noindent First, notice that $W\in T_Z\mathcal{S}$ if and only if $Z^{-1}W\in T_I\mathcal{S}$. If $\textbf{e}_{k}=(0,0,\ldots,0,\underbrace{1}_{{k+1}^{th}},0,\ldots,0,0)^t$, then it is clear that $s(U)=U^t\textbf{e}_{N+1}$. We have to show that $(ds)_Z:T_Z\mathcal{S}\rightarrow T_{s(Z)}\mathbb{S}\lp \mathbb{E}^{\atop N+2}_{\lambda}\rp$ is surjective. In fact, $(ds)_Z(W)=s(W)=W^t\textbf{e}_{N+1}$. Therefore, for a given $V\in T_{s(Z)}\mathbb{S}\lp \mathbb{E}^{\atop N+2}_{\lambda}\rp$, we want to find $W\in T_Z\mathcal{S}$ such that $W^t\textbf{e}_{N+1}=V$, wich is equivalent to $(Z^{-1}W)^t(Z^t\textbf{e}_{N+1})=V$. So we want to find $H\in T_I\mathcal{S}=\mathfrak{s}=\{H\in M_{N+2}(\mathbb{R});H^tG+GH=0\}$ such that $H(Z^t\textbf{e}_{N+1})$. Since $\{Z^t\textbf{e}_{0}\ldots,Z^t\textbf{e}_{N+1}\}$ is a base to $\mathbb{E}^{\atop N+2}_{k}$ and $\pe V,Z^t\textbf{e}_{N+1}\pd=\pe V,s(Z)\pd=0$, this is possible (to see this, change to base $\{Z^t\textbf{e}_{\al}\}$, set the last column of $H$ to being $V$ and construct the rest of $H$ using that $H_{\al\be}=-\va\vb H_{\be\al}$). This finishes the proof. \hfill $\square$
\end{proof}
	
\noindent Now, we prove the following 

\begin{proposition} \label{matrix B}
	
	Let $\lp M,\pe,\pd\rp$ be a semi-Riemannian manifold satisfying the \ce. Set $\mathcal{Z}(x)=\{Z\in \mathcal{S}|Z_{n+1\be}=T_{\be}(x),\be=0,1,\ldots,N+1\}$. Then for each $x_0\in M$ and $B_0\in\mathcal{lZ}(x_0)$, there exists a neighborhood $\mathcal{U}$ of $x_0$ in $M$ and a unique map $B:\mathcal{U}\rightarrow\mathcal{S}$, such that
	
$$
	B^{-1}dB=\Omega-\x,\hspace{0.2cm} B(x_0)=B_0,
$$
and $B(x)\in\mathcal{Z}(x)$, for all $x\in \mathcal{U}.$	
\end{proposition}

\begin{proof}	

For an open neighborhood $\mathcal{U}$ of $x_0\in M$, we define the set 

$$
\mathcal{F}=\{(x,Z)\in \mathcal{U}\times\mathcal{S}|Z\in \mathcal{Z}(x)\}.
$$

\noindent Since $s$ is a submersion by Proposition \ref{submersion s}, the dimension of the manifold $\mathcal{F}$ is 

$$
\text{dim }\mathcal{F}=n+\dfrac{(N+1)(N+2)}{2}-(N+1)=n+\dfrac{N(N+1)}{2}
$$
and its tangent space is

$$
T_{(x,Z)}\mathcal{F}=\{(U,V)\in T_x\mathcal{U}\oplus T_Z\mathcal{S}|(dT_\be)_x(U)=V_{N+1\be},\be=0,\ldots, N+1  \}.
$$
\\
\noindent Consider on $\mathcal{F}$ the distribution $\mathcal{D}(x,Z)=\text{ker }\Theta_{(x,Z)}$, where $\Theta=\Upsilon-Z^{-1}dZ=\Omega-\x-Z^{-1}dZ$. Also consider $\mathcal{H}=\{H\in\mathfrak{s}|ZH\in\text{ker}(ds)_Z\}$. It is easy to see that 
\[
\text{dim }\mathcal{H}=\text{dim }\text{ker}(ds)_Z=\text{dim }\text{ker}(ds)_{I_{N+2}}=\text{dim }\mathfrak{s}-(N+1)=\dfrac{(N+1)(N+2)}{2}-(N+1)=\dfrac{N(N+1)}{2}.
\]
Now, we observe the equivalences

\begin{eqnarray*}
\Theta(U,V)\in\mathcal{H}&\Leftrightarrow&\Omega(U)-\x(U)-Z^{-1}V\in\mathcal{H}\\
&\Leftrightarrow&Z\Omega(U)-Z\x(U)-V\in\text{ker }(ds)_Z\\
&\Leftrightarrow&(Z\Omega(U))_{N+1\be}-(Z\x(U))_{N+1\be}-V_{N+1\be}=0,\text{ for }\be=0,\ldots,N+1 .
\end{eqnarray*}

\noindent But

\begin{eqnarray*}
(Z\Omega(U))_{N+1\be}-(Z\x(U))_{N+1\be}-(dT_{{\be}})_x(U)&=& \\
&=&\ds\sum_{\ga}Z_{N+1\ga}\Omega(U)_{\ga\be}-\ds\sum_{\ga}Z_{N+1\ga}\x(U)_{\ga\be}-(dT_{{\be}})_x(U)\\
&=&\ds\sum_{\ga}T_\ga(x)\w_{\ga\be}(U)-\ds\sum_{\ga}T_\ga(x)\x(U)_{\ga\be}-(dT_{{N+1\be}})_x(U)\\
&=&(dT_{\be})_x(U)-\dfrac{a'}{a}\vb(\wb)_x(U)+\dfrac{\varepsilon  a'}{a}T_{\be}(x)\delta_x(U)\\
&& -\dfrac{\varepsilon  a'}{a}\ds\sum_{\ga}T_\ga(x)\lp T_\be(x)(\wg)_x(U)-\vg\vb T_\ga(x)(\wb)_x(U)\rp\\
&& -(dT_{{\be}})_x(U)\\
&=&0,
\end{eqnarray*}

\noindent for $\be=0,\ldots,N+1$. Therefore, $\text{Im }\Theta\subset\mathcal{H}$. Now, if $H\in\mathcal{l{H}}$, then $(0,-ZH)\in T_{(x,Z)}\mathcal{F}$, which means that $\Theta(0,-ZH)=H$, and it follows that $\Theta$ is surjective, that is, $\text{Im }\Theta=\mathcal{H}$. Hence, we conclude that $\text{dim }\mathcal{D}(x,Z)=\text{dim }T_{(x,Z)}\mathcal{F}-\text{dim Im }\Theta=n.$ Now, we want to prove that $\mathcal{D}$ is integrable. Since $d\Upsilon+\Upsilon\we\Upsilon=0$, we have 
\[
d\Theta=d\Upsilon+Z^{-1}dZ\we Z^{-1}dZ=d\Upsilon+(\Upsilon-\Theta)\we(\Upsilon-\Theta)=\Upsilon\we\Theta-\Theta\we\Upsilon+\Theta\we\Theta.
\] 

\noindent Thus, if $U,V\in\mathcal{D}$, we obtain $d\Theta(U,V)=(\Upsilon\we\Theta-\Theta\we\Upsilon+\Theta\we\Theta)(U,V)=0$. On the other hand, we have $d\Theta(U,V)=U(\Theta(V))-V(\Theta(U))-\Theta([U,V])=-\Theta([U,V])$, and we conclude that $[U,V]\in\mathcal{D}$.

For what is missing, let $L$ be an integral manifold through $(x_0,B_0)$.  For each $(0,V)\in\mathcal{D}_{(x_0,B_0)}$, we have $\Theta_{(x_0,B_0)}(0,V)=B_0^{-1}V=0$, which gives $V=0$. This means that $L$ intersects $\{x_0\}\times\mathcal{S}$ transversely in $(x_0,B_0)$. Shrinking $\mathcal{U}$, if necessary, and using the Local Characterization of Graphs, we conclude that $L$ is a graph of a unique function $B:\mathcal{U}\rightarrow\mathcal{S}.$ Since $L\subset \mathcal{F}$, for each $x\in \mathcal{U}$ we have $B(x)\in  Z(x)$. Finally, since $\Theta\equiv 0$ on $L,\ B$ satisfies the equation $B^{-1}dB=\Omega-\x$. \hfill $\square$
\end{proof}

\section{Proof of the main theorem} \label{PMT}

Let $\mathcal{U}$ be a neighborhood in $M$ of a given point $x_0\in M$ and $B:\mathcal{U}\to\mathcal{S}$ the map found in Proposition \ref{matrix B}. Define $f:\mathcal{U}\rightarrow\varepsilon I\times_a\mathbb{E}_{\lambda}^{N+1}$ by
\[
f_0=cB_{00},\hspace{0.2cm}\ldots\hspace{0.2cm},\hspace{0.2cm}f_{N}=\varepsilon_{N}B_{N0}, \hspace{0.2cm} f_{N+1}=\pi.
\]  

\noindent A direct computation gives us $\text{Im }f \subset\varepsilon I\times_a\mathbb{M}_{\lambda}^{N}(c)$. In order to see that $f$ is an isometric immersion, we compute

\begin{eqnarray*}
df(e_i)&=&\ds\sum_\ga df_\ga(e_i)E_\ga\\
&=&\ds\sum_{\ga\leq N} df_\ga(e_i)E_\ga+d\pi(e_i)E_{N+1}\\
&=&\ds\sum_{\ga\leq N} \vg  dB_{\ga0}(e_i)E_\ga+\varepsilon T_iE_{N+1}\\
&=&\ds\sum_{\ga\leq N} \vg\lp B\Omega(e_i)-B\x(e_i)\rp_{\ga0} E_\ga+\varepsilon_{N+1} B_{N+1i}E_{N+1}\\
&=&\lp\ds\sum_{{\ga\leq N}\atop{\theta}}  \vg B_{\ga\theta}\w_{\theta0}(e_i) E_\ga\rp+\varepsilon_{N+1} B_{N+1i}E_{N+1}\\
&=&\lp\ds\sum_{{\ga\leq N}\atop{\theta}} - \vg B_{\ga\theta}\lp\dfrac{\varepsilon\varepsilon_{\theta}T_iT_{\theta}-\delta_{\theta i}}{ac}\rp E_\ga\rp+\varepsilon_{N+1} B_{N+1i}E_{N+1}\\
&=&\ds\sum_{\ga\leq N}\dfrac{-\vg}{ac} \lp\varepsilon T_i\lp\ds\sum_{\theta}  \varepsilon_{\theta}T_{\theta}B_{\ga\theta}\rp-B_{\ga i}\rp E_\ga+\varepsilon_{ N+1} B_{N+1i}E_{N+1}\\
&=&\ds\sum_{\ga\leq N}\dfrac{-\vg}{ac} \lp\varepsilon T_i\lp\ds\sum_{\theta}  \varepsilon_{\theta}B_{\ga\theta}B_{N+1\theta}\rp-B_{\ga i}\rp E_\ga+\varepsilon_{ N+1} B_{N+1i}E_{N+1}\\
&=&\ds\sum_{\ga\leq N}\dfrac{-\vg}{ac} \lp\varepsilon T_i\vg\delta_{\ga N+1} -B_{\ga i}\rp E_\ga+\varepsilon_{ N+1} B_{N+1i}E_{N+1}\\
&=&\ds\sum_{\ga\leq N}\dfrac{\vg}{ac} B_{\ga i} E_\ga+\varepsilon_{ N+1} B_{N+1i}E_{N+1}\\
&=&\ds\sum_{\ga}\vg B_{\ga i} \overline{E}_\ga.
\end{eqnarray*}	
\noindent Since $B\in\mathcal{S}$, we have

$$
\pe df(e_i),df(e_j)\pd=\pe \ds\sum_{\al}\vg B_{\al i} \overline{E}_\al,\ds\sum_{\ga}\vg B_{\ga j} \overline{E}_\ga\pd=\ds\sum_{\ga} \vg B_{\ga i}B_{\ga j}=\varepsilon_i\delta_{ij}=\pe e_i,e_j\pd,
$$ 
\noindent which means that $f$ is an isometric immersion. Now, let $\tilde{E}_{\ga}=\ds\sum_{\al}\va B_{\al \ga} \overline{E}_\al$. We have that $\E_i=df(e_i)$ is tangent to $f(\mathcal{U})$ and $\E_u$ is normal to $f(\mathcal{U})$. In fact, for each $i$, we have

$$
\pe\E_u,\E_i\pd=\pe \ds\sum_{\al}\va B_{\al u} \overline{E}_\al,\ds\sum_{\ga}\vg B_{\ga i} \overline{E}_\ga\pd=\ds\sum_{\ga} \vg B_{\ga u}B_{\ga i}=\varepsilon_i\delta_{ui}=\pe e_u,e_i\pd=0,
$$

\noindent since $i\neq u$. Then, $Tf(\mathcal{U})^\perp=\text{span }\{\E_u\}$. Let $\Phi$ be a extension of the immersion $f$ defined as follow:

\begin{eqnarray*}
\Phi:T\mathcal{U}
\oplus\sigma^{-1}(\mathcal{U})\subset TM\oplus E&\to&Tf(\mathcal{U})\oplus Tf(\mathcal{U})^\perp\\
\lp p,\ds\sum_\ga \lambda_\ga e_\ga \rp&\to&\lp f(p),\ds\sum_\ga\lambda_\ga\E_\ga\rp,
\end{eqnarray*}

\noindent where $\sigma:E\to M $ is the projection. It is easy to see that $\Phi$ is a isomofism between the bundles $T\mathcal{U}
\oplus\sigma^{-1}(\mathcal{U})$ and  $Tf(\mathcal{U})\oplus Tf(\mathcal{U})^\perp$. Moreover, the metric induced by $\Phi$ also coincides with the given bundle metric since $\Phi(e_\ga)=\E_\ga$. If $\al,\be,\ga\geq1$, then

\begin{eqnarray*}
	\pe\nabla^{\perp}_{\Phi(e_i)}\Phi(e_u),\Phi(e_v)\pd&=&\pe\nt_{\E_i}\E_u,\E_v\pd\\
	&=&\pe\nt_{ _{\ds\sum_{\rho}\varepsilon_\rho B_{\rho i}\e_\rho}}\ds\sum_{\theta}\varepsilon_\theta B_{\theta u}\e_\theta,\ds\sum_{\lambda}\varepsilon_\lambda B_{\lambda v}\e_\lambda\pd\\
		&=&\ds\sum_{\theta}\varepsilon_\theta B_{\theta v} dB_{\theta u}(e_i)+\ds\sum_{\theta,\lambda,\rho}\varepsilon_\theta\varepsilon_\lambda\varepsilon_\rho B_{\rho i}B_{\lambda v}  B_{\theta u}\pe\nt_{\e_\rho}\e_\theta,\e_\lambda\pd\\
		&=&\varepsilon_v(B^{-1}dB(e_ i))_{v u}+\ds\sum_{\theta,\lambda,\rho}\varepsilon_\theta\varepsilon_\lambda\varepsilon_\rho B_{\rho i}B_{\lambda v}  B_{\theta u}\pe\nt_{\e_\rho}\e_\theta,\e_\lambda\pd\\
		&=& \varepsilon_v\w_{vu}(e_i)+\dfrac{\varepsilon a'}{a}B_{N+1 u}\ds\sum_{\theta\leq N}\varepsilon_{\theta}B_{\theta i}B_{\theta v}-\dfrac{\varepsilon a'}{a}B_{N+1 v}\ds\sum_{\theta\leq N}\varepsilon_{\theta}B_{\theta i}B_{\theta u}\\
		&=&\pe e_v,\nabla^{\atop E}_{e_i}e_u\pd=\pe \Phi(e_v),\Phi(\nabla^{\atop E}_{e_i}e_u)\pd.
\end{eqnarray*}

\noindent If we call, by abuse of notation, the isomorfism $\Phi_{|_{\sigma^{-1}(\mathcal{U})}}:\sigma^{-1}(\mathcal{U})\subset E\to Tf(\mathcal{U})^{\perp}$ by $\Phi$ too, then we conclude that

$$
\Phi\nabla^E=\nabla^{\perp}\Phi.
$$

\noindent A similar computation gives $\al_f=\Phi\circ\al^E\circ df^{-1}$, where $\al_f$ is the second fundamental form of $f$. We also notice that 

\begin{eqnarray*}
\partial_t&=&\ds\sum_{\ga}\vg\pe\E_\ga,\partial_t\pd\E_\ga\\&=&\ds\sum_{\ga}\vg\pe\ds\sum_{\theta}\varepsilon_{\theta}B_{\theta\ga}\e_\theta,\partial_t\pd\E_\ga\\
&=&\ds\sum_{\ga}\vg B_{N+1 \ga}\E_\ga\\
&=&\ds\sum_{\ga}\vg T_\ga\E_\ga\\
&=&\ds\sum_{\ga}\vg \pe\eg,T+\xi \pd\E_\ga\\
&=&\ds\sum_{\ga}\vg \pe\Phi(\eg),\Phi(T+\xi)\pd\E_\ga\\
&=&\ds\sum_{\ga}\vg \pe\E_\ga,\Phi(T+\xi)\pd\E_\ga\\
&=&\Phi(T)+\Phi(\xi).
\end{eqnarray*}
Now, we finish the proof by showing that the local immersion $f$ is unique up to a global isometry of $\varepsilon I\times_a\mathbb{M}_{\lambda}^N(c)$. For this, let $\tilde{f}:\tilde{\mathcal{U}}\to \varepsilon I\times_a\mathbb{M}_{\lambda}^N(c)$ be another isometric immersion satisfying all the properties of the theorem \ref{main theorem}, with $\tilde{\mathcal{U}}$ a simply connected neighborhood of $x_0$ included in $\mathcal{U}$ and and $\{\tilde{V}_{\alpha}\}$ the associated frame, with $\tilde{V}_i=d\tilde{f}(e_i)$ and $\tilde{V}_u$ normal to $\tilde{f}(\tilde{U}).$ Let $\tilde{B}$ e be the matrix of the coordinates of the vectors $\tilde{V}_{\alpha}$ in the frame $\{\tilde{E}_{\alpha}\}$. Obviously, up to a direct isometry of $\varepsilon I\times_a\mathbb{M}_{\lambda}^N(c)$, we can assume that $f(x_0) = \tilde{f}(x_0)$ and the frame $\{\tilde{E}_{\alpha}\}$ and $\{\tilde{V}_{\alpha}\}$ coincide at the point $x_0$ and hence $B(x_0)=\tilde{B}(x_0)$. Moreover, these two matrices satisfy all the properties of Proposition \ref{matrix B}, so by uniqueness of the solution in Proposition \ref{matrix B}, we have
$B(x)=\tilde{B}(x)$, for all $x\in\tilde{\mathcal{U}}$. Hence, by construction of $f$ and $\tilde{f}$ from $B$ and $\tilde{B}$, we deduce that $f=\tilde{f}$ in $\mathcal{\tilde{U}}.$\\

Now, we will prove that $f$ can be extended in a unique way to $M$. For this, we consider $x_1$ in $M$ and a curve $\Gamma:[0,1]\to M$ so that $\Gamma(0)=x_0$ and $\Gamma(1)=x_1.$ Then, for each poin $\Gamma(t)$ there exists a neighborhood of $\Gamma(t)$ such that there exists an isometric immersion of this neighbouhood into $\varepsilon I\times_a\mathbb{M}_{\lambda}^N(c)$ satisfying the properties of the theorem. From this family of neighbourhood, we can extract a finite subsequence $(\mathcal{U}_0,\cdots,\mathcal{U}_r)$ covering $\Gamma$ with $\mathcal{U}_0=\mathcal{U}.$ Hence, by uniqueness, we can extend $f$ to $\mathcal{U}_r$ and define $f(x_1)$. We conclude by noting that since $M$ is simply connected, the value of $f(x_1)$ does not depend on the choice of the curve $\Gamma$. Finally, having the global imersion $f:M\to \varepsilon I\times_a\mathbb{M}^N_{\lambda}(c)$ we obtain the desired global bundle isomorphism $\Phi:E\to Tf(M)^{\perp}.$\\

This finishes the proof of Theorem \ref{main theorem}.\\

\noindent {\bf Acknowledgements}

\noindent The first author was partially supported by Conselho
Nacional de Desenvolvimento Científico e Tecnológico (CNPq).


\textsc{Universidade Federal do Piauí - UFPI, Departamento de Ciências Econômicas e Quantitativas, Campus Ministro Reis Velloso, 64000-000, Parnaíba / PI, Brazil.}

\emph{E-mail address}: \texttt{carlos\_aribeiro@yahoo.com.br}\\

\textsc{Universidade Federal do Ceará - UFC, Departamento de Matemática, Campus do Pici, Av. Humberto Monte, Bloco 914, 60455-760, Fortaleza / CE, Brazil.}

\emph{E-mail address}: \texttt{mcosmelo@mat.ufc.br} 


\end{document}